\documentclass[dvipdfmx]{amsart}
\usepackage{amsmath}
\usepackage{amssymb} 
\usepackage{amscd} 
\usepackage{graphicx} 
\usepackage{epsfig}
\usepackage[dvips]{color}
\usepackage{comment}

\newtheorem{theorem}{Theorem}[section]
\newtheorem{lemma}[theorem]{Lemma}

\newtheorem{proposition}[theorem]{Proposition}
\newtheorem{corollary}[theorem]{Corollary}
\newtheorem{claim}[theorem]{Claim}

\newtheorem{remark}[theorem]{Remark}
\newtheorem{question}[theorem]{Question}

\theoremstyle{definition}
\newtheorem{definition}[theorem]{Definition}
\newtheorem{example}[theorem]{Example}

\newtheorem*{thm:finite_family}{Theorem~\ref{thm:finite_family}}
\newtheorem*{thm:peripheral Magnus property}{Theorem~\ref{thm:peripheral Magnus property}}
\newtheorem*{thm:distinct_normal_closure}{Theorem~\ref{distinct_normal_closure}}
\newtheorem*{thm:descending_sequence}{Theorem~\ref{thm:descending_sequence}}
\newtheorem*{thm:ascending_chain}{Theorem~\ref{ascending_chain}}
\newtheorem*{thm:no_infinite_chain}{Theorem~\ref{thm:no_infinite_chain}}

\numberwithin{equation}{section}
\numberwithin{figure}{section}
\numberwithin{table}{section}

\begin{document}
\baselineskip 14pt

\title[Elements in a knot group which are trivialized by Dehn fillings]{Nontrivial elements in a knot group which are trivialized by Dehn fillings}
\author[T. Ito]{Tetsuya Ito}
\address{Department of Mathematics,
Graduate School of Science,
Kyoto University, Kyoto 606-8502, Japan}
\email{tetitoh@math.kyoto-u.ac.jp}
\thanks{The first named author has been partially supported by JSPS KAKENHI Grant Number JP15K17540 and JP16H02145.}

\author[K. Motegi]{Kimihiko Motegi}
\address{Department of Mathematics, Nihon University, 
3-25-40 Sakurajosui, Setagaya-ku, 
Tokyo 156--8550, Japan}
\email{motegi@math.chs.nihon-u.ac.jp}
\thanks{The second named author has been partially supported by JSPS KAKENHI Grant Number JP26400099 and Joint Research Grant of Institute of Natural Sciences at Nihon University for 2017. }

\author[M. Teragaito]{Masakazu Teragaito}
\address{Department of Mathematics and Mathematics Education, Hiroshima University, 
1-1-1 Kagamiyama, Higashi-Hiroshima, 739--8524, Japan}
\email{teragai@hiroshima-u.ac.jp}
\thanks{The third named author has been partially supported by JSPS KAKENHI Grant Number JP16K05149.}

\dedicatory{}

\begin{abstract}
Let $K$ be a nontrivial knot in $S^3$ with the exterior $E(K)$, 
and $\gamma \in G(K) = \pi_1(E(K), *)$ a slope element 
represented by an essential simple closed curve on $\partial E(K)$ with base point $* \in \partial E(K)$. 
Since the normal closure $\langle\!\langle \gamma \rangle\!\rangle$ of $\gamma$ in $G(K)$ 
coincides with that of $\gamma^{-1}$, 
and $\gamma$ and $\gamma^{-1}$ correspond to a slope $r \in \mathbb{Q} \cup \{ \infty \}$,  
we write $\langle\!\langle r \rangle\!\rangle = \langle\!\langle \gamma \rangle\!\rangle$. 
The normal closure $\langle\!\langle r \rangle\!\rangle$ describes elements which are trivialized by $r$--Dehn filling of $E(K)$. 
In this article, 
we prove that 
$\langle\!\langle r_1 \rangle\!\rangle = \langle\!\langle r_2 \rangle\!\rangle$ if and only if $r_1 = r_2$, 
and for a given finite family of slopes $\mathcal{S} = \{ r_1, \dots, r_n \}$, 
the intersection $\langle\!\langle r_1 \rangle\!\rangle \cap \cdots \cap \langle\!\langle r_n \rangle\!\rangle$ contains infinitely many elements except when $K$ is a $(p, q)$--torus knot and $pq \in \mathcal{S}$.  
We also investigate inclusion relation among normal closures of slope elements. 
\end{abstract}

\maketitle

{
\renewcommand{\thefootnote}{}
\footnotetext{2010 \textit{Mathematics Subject Classification.}
Primary 57M05, 57M25
\footnotetext{ \textit{Key words and phrases.}
knot group. Dehn filling, slope element, peripheral Magnus property}
}

\section{Introduction}
\label{Introduction}
Geometric aspects of Dehn fillings such as destroying and creating essential surfaces have been extensively studied by many authors; see survey articles \cite{Go_ICM,Go_Warsaw,G_small,G_Park_City} and references therein. 
In the present article we focus on a group theoretic aspect of Dehn fillings. 
Let $K$ be a nontrivial knot in $S^3$ with its exterior $E(K)$. 
Then by the loop theorem \cite{Papa} the inclusion map $i : \partial E(K) \to E(K)$ induces a monomorphism 
$i_* : \pi_1(\partial E(K), *) \to \pi_1(E(K), *)$, 
where we choose a base point  $*$ in $\partial E(K)$.  
We denote the \textit{knot group} $\pi_1(E(K), *)$ by $G(K)$ and 
its \textit{peripheral subgroup} $i_*(\pi_1(\partial E(K), *))$ by $P(K)$. 
A \textit{slope element} in $G(K)$ is a primitive element $\gamma$ in $P(K) \cong \mathbb{Z} \oplus \mathbb{Z}$, 
which is represented by an essential oriented simple closed curve on $\partial E(K)$ with base point $*$. 
Denote by $\langle\!\langle \gamma \rangle\!\rangle$ the normal closure of $\gamma$ in $G(K)$.  
Taking a standard meridian-longitude pair $(\mu, \lambda)$ of $K$, 
each slope element $\gamma$ is expressed as $\mu^m \lambda^n$ for some relatively prime integers $m, n$.
As usual we use the term \textit{slope} to mean the isotopy class of an essential unoriented simple closed curve on $\partial E(K)$. 
A slope element $\gamma$ and its inverse $\gamma^{-1}$ 
give the same normal subgroup $\langle\!\langle \gamma \rangle\!\rangle = \langle\!\langle \gamma^{-1} \rangle\!\rangle$, 
and by forgetting orientations, 
they correspond to the same slope,  
which may be identified with $m/n \in \mathbb{Q} \cup \{ \infty \}$. 
So in the following we denote 
$\langle\!\langle \gamma \rangle\!\rangle = \langle\!\langle \gamma^{-1} \rangle\!\rangle$ by $\langle\!\langle m/n \rangle\!\rangle$. 
Thus each slope $m/n$ defines the normal subgroup $\langle\!\langle m/n \rangle\!\rangle \subset G(K)$, 
which will be referred to as the \textit{normal closure of the slope $m/n$} for simplicity. 
A slope $r$ is \textit{trivial} if $r = \infty$, i.e. $r$ is represented by a meridian of $K$. 
In what follows, 
we abbreviate the base point for simplicity. 

Denote by $K(r)$ the $3$--manifold obtained from $E(K)$ by $r$--Dehn filling.  
 Then we have the following short exact sequence which relates $G(K),\ \langle\!\langle r \rangle\!\rangle$ and 
$\pi_1(K(r))$. 
\[ \{1\} \rightarrow \langle\!\langle r \rangle\!\rangle \rightarrow G(K) \rightarrow 
G(K) / \langle\!\langle r \rangle\!\rangle = \pi_1(K(r)) \rightarrow \{1\},\]
and thus 
\[
\langle\!\langle r \rangle\!\rangle
 = \{ g \in G(K)\ |\ g\ \textrm{becomes trivial in}\ \pi_1(K(r) \}.
 \]

Recall that a group $G$ possesses the \textit{Magnus property}, 
if whenever two elements $u$, $v$ of $G$ have the same normal closure, 
then $u$ is conjugate to $v$ or $v^{-1}$.  
Magnus \cite{Mag} established this property for free groups, 
and recently \cite{Bo, BS, Fe, H} prove the fundamental groups of closed surfaces have this property. 
However, in general, 
knot groups do not satisfy this property; 
see  \cite{Du, SWW, Tsau0} for details.   
With respect to Dehn fillings, 
the above observation leads us to introduce: 

\begin{definition}[\textbf{peripheral Magnus property}]
\label{peripheral_Magnus_property}
Let $K$ be a nontrivial knot in $S^3$. 
We say that the knot group $G(K)$ has the \textit{peripheral Magnus property} 
if $\langle\!\langle r \rangle\!\rangle = \langle\!\langle r' \rangle\!\rangle$ implies 
$r = r'$ for two slopes $r$ and $r'$. 
\end{definition}

Property P \cite{KM} says that $\langle\!\langle r \rangle\!\rangle = \langle\!\langle \infty \rangle\!\rangle = G(K)$ 
if and only if $r = \infty$. 
We first establish every nontrivial knot group has this property.   

\begin{theorem}
\label{thm:peripheral Magnus property}
For any nontrivial knot $K$, the knot group $G(K)$ satisfies the peripheral Magnus property, 
namely $\langle\!\langle r \rangle\!\rangle = \langle\!\langle r' \rangle\!\rangle$ 
if and only if $r = r'$.
\end{theorem}

When $K$ is a prime, non-amphicheiral knot, 
we will prove a slightly stronger version of Theorem~\ref{thm:peripheral Magnus property}; 
see Theorem \ref{thm:peripheral Magnus property_strong}. 

Theorem~\ref{thm:peripheral Magnus property} says that there is a one to one correspondence between the set of slopes, which is identified with $\mathbb{Q}\cup \{ \infty \}$, 
and the set of normal closures of slopes. 

Next we investigate for which slope $r \in \mathbb{Q} \cup \{ \infty \}$,  
its normal closure $\langle\!\langle r \rangle\!\rangle$ is finitely generated. 
There are two obvious situations where $\langle\!\langle r \rangle\!\rangle$ is finitely generated. 

\begin{itemize}
\item
If $K$ has a \textit{finite surgery slope} $r$, 
i.e. $r$--surgery on $K$ yields a $3$--manifold with finite fundamental group, 
then $\langle\!\langle r \rangle\!\rangle$ is finitely generated. 
(See the proof of Theorem \ref{finitely generated}.)
\item
If $K$ is a torus knot $T_{p, q}$, 
then $\langle\!\langle pq \rangle\!\rangle$ is an infinite cyclic normal subgroup of $G(K)$, hence finitely generated. 
\end{itemize}

Theorem~\ref{finitely generated} below classifies slopes whose normal closures are finitely generated. 

\begin{theorem}
\label{finitely generated}
Let $K$ be a nontrivial knot.
The normal closure $\langle\!\langle r \rangle\!\rangle$ is finitely generated if and only if $r$ is a finite surgery slope, 
or $K$ is a torus knot $T_{p, q}$ and $r = pq$. 
\end{theorem}

Thus generically, 
normal closures of slopes are infinitely generated, 
hence each Dehn filling trivializes an infinitely generated subgroup of $G(K)$. 
So it seems interesting to ask: 
For how many slopes of $K$ do 
their normal closures intersect nontrivially? 
Furthermore, if the intersection is nontrivial, how big is this subgroup?

\begin{theorem}
\label{thm:finite_family}
Let $K$ be a nontrivial knot in $S^3$,  
and let $\{ r_1, \dots, r_n \}$ $(n \geq 2)$ be any finite family of slopes of $K$.  
If $K$ is  a torus knot $T_{p, q}$, 
we assume that $pq \not\in \{ r_1, \dots, r_n \}$. 
Then $\langle\!\langle r_1\rangle\!\rangle \cap \cdots \cap \langle\!\langle r_n\rangle\!\rangle$ is nontrivial.
Moreover, this subgroup is finitely generated if and only if all the $r_i$ are finite surgery slopes. 
\end{theorem}

\begin{remark}
\label{pq_finiteness}
\begin{itemize}
\item[(i)]
If $K = T_{p, q}$ and $pq \in \{ r_1, \dots, r_n \}$, 
then generically $\langle\!\langle r_1\rangle\!\rangle \cap \cdots \cap \langle\!\langle r_n\rangle\!\rangle = \{ 1 \}$; 
see Proposition~\ref{reducing}. 
\item[(ii)]
As the result below shows, the finiteness of a family of slopes in Theorem~\ref{thm:finite_family} is essential. 
\end{itemize}
\end{remark}

\begin{theorem}[\cite{IMT}]
\label{hyperbolic knot}
Let $K$ be a hyperbolic knot in $S^3$. 
Then 
$\langle\!\langle r_1\rangle\!\rangle \cap \langle\!\langle r_2 \rangle\!\rangle \cap \cdots = \{ 1 \}$ 
for any infinite family of slopes $\{ r_1, r_2, \dots \}$. 
\end{theorem}

Combining Theorems~\ref{thm:finite_family} and \ref{hyperbolic knot} we have: 

\begin{corollary}
\label{infinite_finite}
Let $K$ be a hyperbolic knot in $S^3$. 
For any infinite family of slopes, 
their normal closures intersect trivially, 
while for any finite subfamily, 
the intersection of their normal closures contains infinitely many elements. 
\end{corollary}

This has the following interpretation from a viewpoint of Dehn fillings. 

\begin{corollary}
\label{infinite_finite_geometric}
Let $K$ be a hyperbolic knot in $S^3$. 
Then for any nontrivial element $g \in G(K)$, 
there are only finitely many Dehn fillings of $E(K)$ which trivialize $g$, 
while for any finitely many Dehn fillings, 
there are infinitely many nontrivial elements which become trivial by each of these Dehn fillings. 
\end{corollary}

\begin{remark}
\begin{itemize}
\item[(i)]
Let $K$ be a nontrivial torus knot $T_{p, q}$.  
Then 
$\displaystyle\cap_{n \in \mathbb{Z}}\langle\!\langle (pqn \pm 1)/n \rangle\!\rangle = [G(K), G(K)]$, 
which is the free group of rank $(p-1)(q-1)/2$. See Proposition~\ref{cyclic_commutator}. 

\item[(ii)]
For satellite knots, it is still open, but we expect $\langle\!\langle r_1\rangle\!\rangle \cap \langle\!\langle r_2 \rangle\!\rangle \cap \cdots = \{ 1 \}$ for any infinite family of slopes $\{ r_1, r_2, \dots \}$.  
See \cite{IMT_g-torison_filling} for further discussion. 
\end{itemize}
\end{remark}

Let us turn to inclusion relations among normal closures of slopes. 
Since $\langle\!\langle \infty \rangle\!\rangle = G(K)$, 
$\langle\!\langle r \rangle\!\rangle \subset \langle\!\langle \infty \rangle\!\rangle$ for any slope $r$. 

\begin{theorem}
\label{thm:no_infinite_chain}
Let $K$ be a non-torus knot in $S^3$. 
If $\langle\!\langle r_1 \rangle\!\rangle \supset \cdots \supset \langle\!\langle r_n \rangle\!\rangle$
for mutually distinct slopes $r_1,r_2,\dots, r_n \in \mathbb{Q}$, 
then $n \le 2$.  
In particular, there is no infinite descending chain nor ascending chain of normal closures of slopes. 
\end{theorem}

On the contrary, for torus knots we have: 

\begin{theorem}
\label{thm:descending_sequence}
Let $K$ be a torus knot $T_{p,q}\ (p > q\ge 2)$. 
\begin{itemize}
\item[(i)]
There is no infinite ascending chain 
$\langle\!\langle r_1 \rangle\!\rangle \subset  \langle\!\langle r_2 \rangle\!\rangle \subset \langle\!\langle r_3 \rangle\!\rangle \subset \cdots$. 

\item[(ii)]
For each finite surgery slope $r$,
there exists an infinite descending chain
\[
\langle\!\langle r \rangle\!\rangle \supset \langle\!\langle r_1 \rangle\!\rangle\supset 
\langle\!\langle r_2 \rangle\!\rangle\supset \cdots.
\]
\end{itemize}
\end{theorem}

\section{Inclusions between two normal closures of slopes}
\label{inclusions}

In this section we study inclusions between two normal closures of slopes. 
We say that a slope $r$ is a  \textit{reducing surgery slope} if $K(r)$ is a reducible 3-manifold.

\begin{lemma}
\label{power of meridian}
Let $K$ be a nontrivial knot in $S^3$ with meridian $\mu$. 
If $\mu^a \in \langle\!\langle r \rangle\!\rangle$ for some integer $a \ne 0$,  
then $r$ is a finite surgery slope or a reducing surgery slope.
\end{lemma}

\begin{proof}
Without loss of generality we may assume $a > 0$.  
If $a = 1$, 
then $\mu \in \langle\!\langle r \rangle\!\rangle$ and $\pi_1(K(r)) = G(K) / \langle\!\langle r \rangle\!\rangle =\{ 1 \}$, 
and thus $r$ is a finite surgery slope. 
(Actually, Property P \cite{KM} implies $r = \infty$.)  
So in the following we assume $a \ge 2$. 
Since $\mu^a \in \langle\!\langle r \rangle\!\rangle$, 
$\langle\!\langle \mu^a \rangle\!\rangle \subset \langle\!\langle r \rangle\!\rangle$  
and we have the canonical epimorphism 
$\varphi : G(K) / \langle\!\langle \mu^a \rangle\!\rangle \to G(K) / \langle\!\langle r \rangle\!\rangle$. 
Note that $(\varphi(\mu))^a = \varphi(\mu^a) = 1$ in $G(K) / \langle\!\langle r \rangle\!\rangle$. 
If $\varphi(\mu) = 1 \in G(K) / \langle\!\langle r \rangle\!\rangle$,  
then $\varphi(\mu) = \mu \in  \langle\!\langle r \rangle\!\rangle$ and as above $r = \infty$. 
Thus we may assume $\varphi(\mu) \ne 1$, i.e. it is a nontrivial torsion element in $G(K)\slash \langle\!\langle r \rangle\!\rangle$.  
Recall that an irreducible $3$--manifold $M$ with infinite fundamental group is aspherical \cite[p.48 (C.1)]{AFW}
and hence $\pi_1(M)$ has no torsion element \cite{Hem}. 
Hence $\pi_1(K(r))$ is finite or $K(r)$ must be a reducible manifold. 
Accordingly $r$ is a finite surgery slope or a reducing surgery slope. 
\end{proof}

\begin{lemma}
\label{non_reducing}
Let $K$ be a nontrivial knot in $S^3$ with meridian $\mu$, 
and $r$ a nontrivial slope. 
If $\mu^a \in \langle\!\langle r \rangle\!\rangle$ for some integer $a \ne 0$,  
then $r$ is not a reducing surgery slope. 
\end{lemma}

\begin{proof}
Assume to the contrary that $r$ is a reducing surgery slope. 
As in the proof of Lemma~\ref{power of meridian} $G(K) / \langle\!\langle r \rangle\!\rangle$ has a nontrivial torsion element, 
hence $K(r) \ne S^2 \times S^1$ and thus $K(r)$ is a connected sum of two closed $3$--manifolds other than $S^3$. 
(In general, 
the result of a surgery on a nontrivial knot is not $S^2 \times S^1$ \cite{GabaiIII}.)
By the Poincar\'e conjecture, they have nontrivial fundamental groups, 
and $G(K)\slash \langle\!\langle r \rangle\!\rangle = A \ast B$ for some nontrivial groups $A$ and $B$. 

As we have seen in the proof of Lemma \ref{power of meridian}, $\varphi(\mu)$, 
the image of $\mu$ under the canonical epimorphism 
$\varphi : G(K) / \langle\!\langle \mu^a \rangle\!\rangle \to G(K) \slash \langle\!\langle r \rangle\!\rangle$ 
is a nontrivial torsion element in $A \ast B$. 
By \cite[Corollary~4.1.4]{MKS}, 
a nontrivial torsion element in a free product $A \ast B$ is conjugate to a torsion element of $A$ or $B$. 
Thus we may assume that there exists $g \in A\ast B$ such that $g\varphi(\mu)g^{-1} \in A$. 

On the other hand, 
$A \ast B$ is normally generated by $\varphi(\mu)$ since $G(K)$ is normally generated by $\mu$. 
This implies that $A \ast B$ is normally generated by an element $g\varphi(\mu)g^{-1} \in A$. 
In particular, 
the normal closure $\langle\!\langle A \rangle\!\rangle$ of $A$ in $A \ast B$ is equal to $A \ast B$, 
and $(A \ast B) \slash \langle\!\langle A \rangle\!\rangle = \{ 1 \}$. 
However, $(A \ast B) \slash \langle\!\langle A \rangle\!\rangle = B \neq \{1\}$ (\cite[p.194]{MKS}). 
This is a contradiction. 
\end{proof}

Combine Lemmas~\ref{power of meridian} and \ref{non_reducing} to obtain the following result 
which asserts that inclusions among normal closures of slopes are quite limited.

\begin{proposition}
\label{inclusion}
Let $K$ be a nontrivial knot in $S^3$. 
Assume that $\langle\!\langle r \rangle\!\rangle \supset \langle\!\langle r' \rangle\!\rangle$ for distinct slopes $r$ and $r'$. 
Then $r$ is a finite surgery slope.  
\end{proposition}

\begin{proof}
If $r=\infty$, there is nothing to prove.  
Hence we assume $r \ne \infty$.
Write $r = m/n$ and $r' = m'/n'$. 
By the assumption $\mu^{m'} \lambda^{n'} \in \langle\!\langle m/n \rangle\!\rangle$, and hence 
\[\mu^{-mn' + nm'} = (\mu^m \lambda^n)^{-n'} (\mu^{m'} \lambda^{n'})^{n} \in \langle\!\langle m/n \rangle\!\rangle.\]
Since $r$ and $r'$ are distinct slopes, $-mn' + nm' \ne 0$. 
Then Lemma~\ref{power of meridian} shows that $r$ is a finite surgery slope or a reducing surgery slope. 
However, the latter case cannot happen by Lemma~\ref{non_reducing}.  
\end{proof}

\begin{proposition}
\label{non_finite_finite}
Let $K$ be a non-torus knot in $S^3$. 
Assume that $\langle\!\langle r \rangle\!\rangle \supset \langle\!\langle r' \rangle\!\rangle$ 
for distinct nontrivial slopes $r$ and $r'$. 
Then $r'$ is not a finite surgery slope.  
\end{proposition}

\begin{proof}
By the assumption and Proposition~\ref{inclusion}, 
$r$ is a nontrivial finite surgery slope. 
Assume to the contrary that $r'$ is also a finite surgery slope.
Write $r = m/n$ and $r' = m'/n'$ with $m, m'>0$. 
Since $\langle\!\langle m/n \rangle\!\rangle \supset \langle\!\langle m'/n' \rangle\!\rangle$, 
we have a canonical epimorphism 
from $G(K) / \langle\!\langle m'/n' \rangle\!\rangle$ to $G(K) / \langle\!\langle m/n \rangle\!\rangle$, 
which induces an epimorphism 
\[\mathbb{Z}_{m'} \cong H_1(K(m'/n')) \to H_1(K(m/n)) \cong \mathbb{Z}_{m}.\] 
This then implies that $m' \ge m$ and $m'$ is a multiple of $m$. 

By the assumption $K$ is a hyperbolic knot or a satellite knot. 
Furthermore, in the latter case, 
since $K$ admits a nontrivial finite surgery, 
$K$ is a $(p, q)$--cable of a torus knot $T_{p', q'}$, where $|p| \ge 2$ \cite{BZ_finite_JAMS}.

\noindent
\textbf{Case 1.}
$K$ is a hyperbolic knot. 

Recall that the distance between finite surgery slopes of a hyperbolic knot is at most two \cite{NZ_finite}.
Hence $|mn' - nm' | \leq 2$. 
Since $m'$ is a multiple of $m$, 
the inequality $|mn' - nm' | \leq 2$ implies $m= 1,2$. 
Since $n \ne 0$, 
we have $|\frac{m}{n}| \leq 2$. 
A finite surgery is also an L-space surgery, 
so by \cite[Corollary 1.4]{OS_rational}, $2\geq |\frac{m}{n}|\geq 2g(K)-1$, 
where $g(K)$ is the genus of $K$. 
This implies $g(K) = 1$.  
Since a knot admitting an L-space surgery is fibered \cite{Ni, Ni2, Ghig, Juh}, 
$K$ is a trefoil knot $T_{3, 2}$ (or $T_{-3, 2}$) or the figure-eight knot. 
By assumption, $K$ is not a torus knot,  
and hence $K$ would be the figure-eight knot. 
However, the figure-eight knot has no nontrivial finite surgery, 
a contradiction.

\noindent
\textbf{Case 2.} 
$K$ is a $(p, q)$--cable of a torus knot $T_{p', q'}$. 

Finite surgeries on iterated torus knots are classified by \cite[Table~1]{BH}. 
For any cable of a torus knot which admits two finite surgeries $m/n$ and $m'/n'$, 
$m'$ is not a multiple of $m$.

This completes a proof of Proposition~\ref{non_finite_finite}. 
\end{proof}

\section{Peripheral Magnus property for knot groups}
\label{Magnus}

Now we are ready to prove the peripheral Magnus property for knot groups. 
We separate the proof into two cases depending upon $K$ is a non-torus knot or $K$ is a torus knot. 
Furthermore, in the latter case, 
we distinguish the specific case where $K$ is a trefoil knot $T_{3, 2}$ and 
$r = (18k+9) / (3k+1)$, $r' = (18k+9) / (3k+2)$ for technical reasons; see Proposition~\ref{trefoil}. 
These surgeries are only examples of surgeries along torus knots which give rise to (orientation reversingly) homeomorphic 3-manifolds with finite fundamental group \cite{Mathieu}. 
Thus, the case treated in Proposition \ref{trefoil} gives us the most subtle situation. 

\begin{thm:peripheral Magnus property}
Any nontrivial knot satisfies the peripheral Magnus property, 
namely $\langle\!\langle r \rangle\!\rangle = \langle\!\langle r' \rangle\!\rangle$ 
if and only if $r = r'$.
\end{thm:peripheral Magnus property}

\begin{proof}
The ``if'' part is obvious. 
Let us prove the ``only if'' part. 
Recall first that if $r= \infty$, 
then $\langle\!\langle r' \rangle\!\rangle =  \langle\!\langle \infty \rangle\!\rangle$ can happen only when $r' = \infty$ by Property P \cite{KM}. 
If $r = 0$, 
then $\langle\!\langle r' \rangle\!\rangle = \langle\!\langle 0 \rangle\!\rangle$ holds only if $r' = 0$ for homological reason. 
So in the following we assume $r$ is neither $\infty$ nor $0$. 
We divide the argument into two cases depending upon $K$ is a torus knot or a non-torus knot.

\noindent
\textbf{Case 1.}\ $K$ is a non-torus knot. 

Suppose for a contradiction that we have mutually distinct slopes $r$ and $r'$ which satisfy  
$\langle\!\langle r \rangle\!\rangle = \langle\!\langle r' \rangle\!\rangle$. 
Since $\langle\!\langle r \rangle\!\rangle \supset \langle\!\langle r' \rangle\!\rangle$, 
Propositions~\ref{inclusion} and \ref{non_finite_finite} show that $r$ is a finite surgery slope and $r'$ is not a finite surgery slope. 
On the other hand, 
since $\langle\!\langle r \rangle\!\rangle \subset \langle\!\langle r' \rangle\!\rangle$, 
Proposition~\ref{inclusion} implies that $r'$ is a finite surgery slope, 
a contradiction. 

\noindent
\textbf{Case 2.}\ $K$ is a torus knot $T_{p, q}$. 

Without loss of generality, we assume $p > q \ge 2$. 
Assume that $\langle\!\langle r \rangle\!\rangle = \langle\!\langle r' \rangle\!\rangle$ 
for mutually distinct slopes $r$ and $r'$. 
By Proposition~\ref{inclusion} $r$ and $r'$ are finite surgery slopes. 
Let us write $r = m/ n$ and $r' = m' / n'$; 
we may assume $m, m' > 0$.  
Then $m = |H_1(K(m/n))| =  |H_1(K(m'/n'))| = m'$. 
Recall that $T_{p, q}(m/n)$ has a Seifert fibration with base orbifold $S^2(p, q, |pqn-m|)$. 

Assume first that $\pi_1(T_{p, q}(m/n))$ is cyclic, 
i.e. $|pqn-m| = 1$. 
Then since $\pi_1(T_{p, q}(m/n')) \cong \pi_1(T_{p, q}(m/n))$ is finite cyclic, 
we have $|pqn-m| = |pqn' - m| =1$. 
A simple computation shows that 
$|n - n'| = \frac{2}{pq}$, 
which is impossible, because $pq \ge 6$.

Suppose next that $\pi_1(T_{p, q}(m/n))$ is finite, but non-cyclic.  
Then $\{ p, q, |pqn-m| \} = \{ 2, 2, A \}$ (where $A \ge 3$ is an odd integer),  
$\{2, 3, 3\}, \{ 2, 3, 4 \}$ or $\{ 2, 3, 5 \}$.

\noindent
\textbf{Subcase 1.}\ 
Assume that $\{ p, q, |pqn-m| \} = \{ 2, 2, A \}$. 
Then $K$ is a torus knot $T_{2, A}$ and $|2An - m| = 2$. 
Since $\pi_1(T_{2,A}(m/n)) \cong \pi_1(T_{2,A}(m/n'))$, 
$T_{2,A}(m/n)$ and $T_{2,A}(m/n')$ are homeomorphic \cite{AFW}. 
Thus $T_{A, 2}(m/n')$ has a base orbifold $S^2(2, 2, A)$, where $|2An' - m| = 2$. 
By the assumption $n \ne n'$, 
and hence 
$|n-n'| = \frac{2}{A}$, which is an integer. 
This means $A = 1$ or $2$, a contradiction.

\noindent
\textbf{Subcase 2.}\ 
Assume that  $\{ p, q, |pqn-m| \} = \{ 2, 3, 3 \}$. 
Then $K = T_{3, 2}$ and $|6n-m| = 3$. 
Similarly, we have $|6n' - m| = 3$. 
Thus $m$ is a multiple of $3$ and since $n$ and $n'$ are coprime to $m$, 
neither $n$ nor $n'$ is divided by $3$. 
Furthermore, the equalities $|6n-m| = |6n' - m| = 3$ gives 
$|n - n'| = 1$. 
Write $n = 3k +1$ for some integer $k$. 
Then $n' = 3k+2$ and 
$|6n-m| = |6n' - m| = 3$ is written as 
$|18k +6-m| = |18k+12 -m| = 3$. 
Hence we have a unique solution $m = 18k +9$. 
This gives 
$m/n = (18k+9)/(3k+1)$, $m/n' = (18k+9)/(3k+2)$. 
However, in this case 
$\langle\!\langle (18k+9)/(3k+1) \rangle\!\rangle \ne  \langle\!\langle (18k+9)/(3k+2) \rangle\!\rangle$ by Proposition~\ref{trefoil} below.

\noindent
\textbf{Subcase 3.}\ 
Assume that $\{ p, q, |pqn-m| \} = \{ 2, 3, 4 \}$. 
Then we have two possibilities: 
$K = T_{3, 2}$ and $|6n-m| = |6n' - m| = 4$, 
or $K = T_{4, 3}$ and $|12n-m| = |12n' - m| = 2$. 
In the former case, 
$|n - n'| = \frac{4}{3} \not\in \mathbb{Z}$, a contradiction.  
In the latter case, 
$|n - n'| = \frac{1}{3} \not\in \mathbb{Z}$, a contradiction.

\noindent
\textbf{Subcase 4.}\ 
Assume that $\{ p, q, |pqn-m| \} = \{ 2, 3, 5 \}$. 
Then we have three possibilities: 
$K = T_{3, 2}$ and $|6n-m| = |6n' -m| = 5$, 
$K = T_{5, 3}$ and $|15n-m| = |15n' -m| = 2$, 
or $K = T_{5, 2}$ and $|10n-m| = |10n' -m| = 3$. 
In either case $|n - n'|$ cannot be an integer and we have a contradiction.

This completes a proof of Theorem~\ref{thm:peripheral Magnus property}.
\end{proof}

\begin{proposition}
\label{trefoil}
Let $K$ be the trefoil knot $T_{3,2}$.
In the knot group $G(K)$, 
\[
\langle\!\langle (18k+9) / (3k+1) \rangle\!\rangle \ne \langle\!\langle (18k+9) / (3k+2)\rangle\!\rangle.
\]
More precisely, 
$\langle\!\langle (18k+9) / (3k+1) \rangle\!\rangle \not\subset \langle\!\langle (18k+9) / (3k+2) \rangle\!\rangle$
and 
$\langle\!\langle (18k+9) / (3k+1) \rangle\!\rangle \not\supset \langle\!\langle (18k+9) / (3k+2) \rangle\!\rangle$.
\end{proposition}

\begin{proof}
Figure \ref{fig:trefoil} shows that $K(r+6)$ is a Seifert fibered manifold with Seifert invariant 
$S^2(1/r, -1/3, -1/2, 1)$. 

\begin{figure}
\centering
\includegraphics[width=0.8\textwidth]{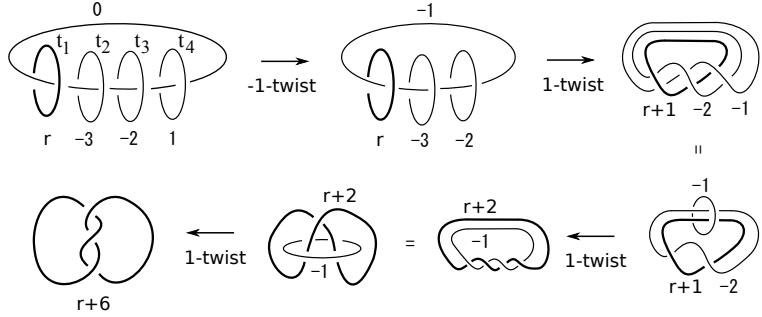}
\caption{Surgery diagrams for $K(r+6)$}\label{fig:trefoil}
\end{figure}

Let us choose $r = 3/(3k+1)$ so that $r+6 = (18k+9)/(3k+1)$. 
Then $K((18k+9)/(3k+1))$ has a Seifert invariant:
\[S^2((3k+1)/3, -1/3, -1/2, 1) = S^2(k; 1/3, -1/3, 1/2).\]
Similarly if we choose $r = -3/(3k+2)$, 
then $r + 6 = (18k+9)/(3k+2)$, 
and $K((18k+9)/(3k+2))$ has a Seifert invariant:
\[S^2((-3k-2)/3, -1/3, -1/2, 1) = S^2(-k; 1/3, -1/3, -1/2).\]
This shows that $K((18k+9)/(3k+1))$ is orientation reversingly homeomorphic to 
$K((18k+9)/(3k+2))$ \cite{Mathieu}.

To obtain a presentation of their fundamental groups, 
we fix a section for the circle bundle 
$K((18k+9)/(3k+1)) - \bigcup_{i = 1}^4 N(t_i)
= K((18k+9)/(3k+2)) - \bigcup_{i = 1}^4 N(t_i)$ 
arising from the top left picture of Figure~\ref{fig:trefoil}, 
where $N(t_i)$ is a fibered tubular neighborhood. 
Note that $t_1$ is the surgery dual to $K$ and $t_4$ is a regular fiber.

Then, with this section $\pi_1(K((18k+9)/(3k+1)))$ has a presentation:  
\[\langle c_1,c_2,c_3,c_4, h \mid 
[c_i,h]=1, c_1c_2c_3c_4=1, 
c_1^3 h^{3k+1}=1, 
c_2^3h^{-1}=1, 
c_3^2h^{-1}=1, 
c_4h = 1\rangle,\]
where $c_i$ is represented by a meridian of $t_i$, a boundary of the section, 
and $h$ is represented by a regular fiber. 

Using the same bases $c_1, \dots, c_4, h$, 
$\pi_1(K((18k+9)/(3k+2)))$ has a presentation: 
\[\langle c_1,c_2,c_3,c_4, h \mid 
[c_i,h]=1, c_1c_2c_3c_4=1, 
c_1^3 h^{-3k-2}=1, 
c_2^3h^{-1}=1, 
c_3^2h^{-1}=1, 
c_4h = 1\rangle.\]

Note that $|\pi_1(K((18k+9)/(3k+1)))| = |\pi_1(K((18k+9)/(3k+2)))|$ is 
$24m$ for some integer $m \ge 1$ coprime to $6$ \cite{LR, Or}.  

Note that the element $h$ is central and generates a cyclic normal subgroup $\langle h \rangle$ in 
both $\pi_1(K((18k+9)/(3k+1)))$ and $\pi_1(K((18k+9)/(3k+2)))$. 
Let us consider the quotient groups 
$\pi_1(K((18k+9)/(3k+1))) / \langle h \rangle$ and 
$\pi_1(K((18k+9)/(3k+2))) / \langle h \rangle$, 
which have the same presentation:
\[\langle c_1,c_2,c_3\mid c_1^3=c_2^3=c_3^2=c_1c_2c_3=1\rangle.\]
This group is the tetrahedral group (spherical triangle group $\Delta(2, 3, 3)$) of order $12$. 
Hence $h$ has order $24m/12 = 2m$ in both $\pi_1(K((18k+9)/(3k+1))$ and 
$\pi_1(K((18k+9)/(3k+2))$. 

\begin{claim}
\label{(18k+9)/(3k+2)_not_in_(18k+9)/(3k+1)}
The slope element with slope $(18k+9)/(3k+2)$ does not belong to $\langle\!\langle (18k+9)/(3k+1)\rangle\!\rangle$. 
\end{claim}

\begin{proof}
We first observe that the slope element with slope $(18k+9)/(3k+2)$ is expressed as 
$c_1^{3} h^{-3k-2}$; see Figure~\ref{fig:trefoil}. 
In $\pi_1(K((18k+9)/(3k+1)))$, 
the above presentation shows that $c_1^3 h^{3k+1}=1$. 
Hence $c_1^3 h^{-3k-2} = h^{-6k-3}$. 
Assume for a contradiction that 
$c_1^3 h^{-3k-2} = h^{-6k-3} \in \langle\!\langle (18k+9)/(3k+1) \rangle\!\rangle$. 
Then $h^{6k+3} = 1$ in $\pi_1(K((18k+9)/(3k+1)))$. 
Hence $6k+3$ would be divided by $2m$. 
This is impossible. 
\end{proof}

Similarly we have: 

\begin{claim}
\label{(18k+9)/(3k+1)_not_in_(18k+9)/(3k+2)}
The slope element with slope $(18k+9)/(3k+1)$ does not belong to $\langle\!\langle (18k+9)/(3k+2)\rangle\!\rangle$. 
\end{claim}

\begin{proof}
We first observe that the slope element with slope $(18k+9)/(3k+1)$ is expressed as 
$c_1^{3} h^{3k+1}$; see Figure~\ref{fig:trefoil}. 
In $\pi_1(K((18k+9)/(3k+2)))$, 
the above presentation shows that $c_1^3 h^{-3k-2}=1$. 
Hence $c_1^3 h^{3k+1} = h^{6k+3}$. 
Assume for a contradiction that 
$c_1^{3} h^{3k+1} = h^{6k+3}$ 
belongs to $\langle\!\langle (18k+9)/(3k+2) \rangle\!\rangle$. 
Then $h^{6k+3} = 1$ in $\pi_1(K((18k+9)/(3k+2)))$, 
and we have a contradiction. 
\end{proof}

Claims~\ref{(18k+9)/(3k+2)_not_in_(18k+9)/(3k+1)} and \ref{(18k+9)/(3k+1)_not_in_(18k+9)/(3k+2)} 
shows that $\langle\!\langle (18k+9)/(3k+1) \rangle\!\rangle \neq \langle\!\langle (18k+9)/(3k+2) \rangle\!\rangle$.
\end{proof}

Let $r$ and $r'$ be distinct slopes. 
Then Theorem~\ref{thm:peripheral Magnus property} says that 
$\langle\!\langle r \rangle\!\rangle \ne \langle\!\langle r' \rangle\!\rangle$. 
However, 
two normal subgroups $\langle\!\langle r \rangle\!\rangle$ 
and $\langle\!\langle r' \rangle\!\rangle$ may be ``similar'' in the sense that there is an automorphism $\varphi : G(K) \to G(K)$ such that 
$\varphi(\langle\!\langle r \rangle\!\rangle) = \langle\!\langle r' \rangle\!\rangle$. 
If $\langle\!\langle r \rangle\!\rangle$ and $\langle\!\langle r' \rangle\!\rangle$ are similar, 
then the automorphism $\varphi : G(K) \to G(K)$ induces an isomorphism 
$G(K)/\langle\!\langle r_1 \rangle\!\rangle \to  G(K)/\langle\!\langle r_2 \rangle\!\rangle$. 
For instance, 
if $K$ is amphicheiral, 
i.e. $E(K)$ admits an orientation revering homeomorphism $f$, 
then $f$ induces an automorphism $f_* : G(K) \to G(K)$ such that 
$f_*(\langle\!\langle r \rangle\!\rangle) = \langle\!\langle -r \rangle\!\rangle$, 
hence $\langle\!\langle r \rangle\!\rangle$ and $\langle\!\langle -r \rangle\!\rangle$ are distinct, but similar.

The next result shows that normal closures of slopes are rigid under automorphisms of $G(K)$ for prime, non-amphicheiral knots $K$. 

\begin{proposition}
\label{invariance}
Let $K$ be a prime knot and $\varphi$ an automorphism of $G(K)$. 
Then for any slope $r \in \mathbb{Q}$, 
$\varphi(\langle\!\langle r \rangle\!\rangle) = \langle\!\langle r \rangle\!\rangle$ or 
$\langle\!\langle -r \rangle\!\rangle$. 
Furthermore, 
if $K$ is not amphicheiral, 
then $\varphi(\langle\!\langle r \rangle\!\rangle) = \langle\!\langle r \rangle\!\rangle$. 
\end{proposition}

\begin{proof}
When $K$ is a prime knot, \cite[Corollary 4.2]{Tsau} shows that any automorphism of $G(K)$ is induced by a homeomorphism of $E(K)$. 
Let $(\mu,\, \lambda)$ be a standard meridian-longitude pair of $K$. 
Since $\varphi(\mu) = \mu^{\varepsilon }$ ($\varepsilon = \pm 1$) by \cite{G-Lu2} and 
$\varphi(\lambda) = \lambda^{ \varepsilon'}$ ($ \varepsilon' = \pm 1$) for homological reasons, 
$\varphi(\mu^m\lambda^n) = \mu^{\varepsilon m}\lambda^{\varepsilon' n}$. 
This implies that $\varphi(\langle\!\langle m/n \rangle\!\rangle) 
= \langle\!\langle m/n \rangle\!\rangle$ or 
$ \langle\!\langle -m/n \rangle\!\rangle$. 
If $K$ is non-amphicheiral, 
then the homeomorphism preserves the orientation of $S^3$, and hence $\varepsilon = \varepsilon'$. 
Thus $\varphi(\langle\!\langle m/n \rangle\!\rangle) = \langle\!\langle m/n \rangle\!\rangle$.
\end{proof}

Theorem~\ref{thm:peripheral Magnus property}, 
together with Proposition~\ref{invariance}, 
implies the following stronger version of the peripheral Magnus property for prime, non-amphicheiral knots.

\begin{theorem}
\label{thm:peripheral Magnus property_strong}
Let $K$ be a nontrivial, prime knot. 
If $\varphi( \langle\!\langle r \rangle\!\rangle) = \langle\!\langle r' \rangle\!\rangle$ 
for some automorphism $\varphi$ of $G(K)$, 
then $r' = \pm r$. 
Moreover, if $K$ is not amphicheiral or $\varphi=id$, then $r' = r$.
\end{theorem} 

\begin{proof}
Following Proposition~\ref{invariance} 
$\langle\!\langle r' \rangle\!\rangle = \varphi( \langle\!\langle r \rangle\!\rangle)$ 
is either $\langle\!\langle r \rangle\!\rangle$ or $\langle\!\langle -r \rangle\!\rangle$; 
the latter can happen only when $K$ is amphicheiral. 
Then the peripheral Magnus property (Theorem~\ref{thm:peripheral Magnus property}) shows 
$r' = r$ or $-r$, respectively. 
\end{proof}

As we have mentioned, for the torus knot $T_{3, 2}$ and slopes 
$r=(18k+9)/(3k+1)$ and $r'= (18k+9)/(3k+2)$, $T_{3, 2}(r)$ is (orientation reversingly) homeomorphic to $T_{3, 2}(r')$  \cite{Mathieu}. 
Thus the quotients $G(T_{3, 2})/ \langle\!\langle r \rangle\!\rangle $ and $G(T_{3, 2})/\langle\!\langle r' \rangle\!\rangle$ are isomorphic. 
On the other hand, 
Theorem \ref{thm:peripheral Magnus property_strong} says that no isomorphism between 
$G(T_{3, 2})/ \langle\!\langle r \rangle\!\rangle$ and $G(T_{3, 2})/\langle\!\langle r' \rangle\!\rangle$ can be induced by an automorphism of $G(T_{3,2})$. 
This illustrates a possibility for non-similar normal subgroups yielding the same quotient groups. 

\medskip

We close this section by giving the following observation for conjugacy among slope elements. 
A knot $K$ is \textit{cabled} if it is a $(p,q)$-cable of a knot $K'$, 
and we call the slope $pq$ the \textit{cabling slope} of a cabled knot $K$. In the following, 
we regard a torus knot as a cabled knot, the $(p,q)$-cable of the unknot.
If $K$ is neither a cabled knot nor a composite knot, 
then the peripheral subgroup $P(K)$ is malnormal \cite{HW}, 
and hence if $g \gamma_1 g^{-1} = \gamma_2$, 
then $g$ belongs to $P(K)$ and $\gamma_1 = \gamma_2$. 
Even when $K$ is a cabled knot or a composite knot we have: 

\begin{proposition}[conjugation of slope elements]
\label{conjugation}
Let $K$ be a nontrivial knot, 
and let $\gamma_1$ and $\gamma_2$ be slope elements in $P(K)$. 
Assume that $g \gamma_1 g^{-1} = \gamma_2$ in $G(K)$.  
Then $\gamma_2 = \gamma_1$. 
Furthermore, if $g \not\in P(K)$, 
then $K$ is a cabled knot or a composite knot and $\gamma_1$ represents the cabling slope or the meridional slope, respectively. 
\end{proposition}

\begin{proof}
By assumption 
$\langle\!\langle \gamma_2 \rangle\!\rangle  = \langle\!\langle g \gamma_1 g^{-1} \rangle\!\rangle$, 
which coincides with $\langle\!\langle \gamma_1 \rangle\!\rangle$. 
Then the peripheral Magnus property (Theorem~\ref{thm:peripheral Magnus property}) says $\gamma_2 = \gamma_1$ or $\gamma_1^{-1}$. 
If $g \in P(K)$, then $g$ commutes with $\gamma_1$ hence $\gamma_2 = g \gamma_1 g^{-1} = \gamma_1$. So we assume $g \not\in P(K)$. 
Then we have a non-degenerate map $f : S^1 \times [0, 1] \to E(K)$ 
such that $f(S^1 \times \{ 0 \}) = c$ representing $\gamma_1$ and 
$f(S^1 \times \{ 1 \}) = c'$ representing $\gamma_2$.
Since $\gamma_2 = \gamma_1$ or $\gamma_1^{-1}$, 
we may assume $c \cap c' = \emptyset$. 
By \cite[VIII.13.Annulus Theorem]{Ja} we have a non-degenerate embedding $g : S^1 \times [0, 1] \to E(K)$ 
with $g |_{S^1 \times \{0, 1\}} = f |_{S^1 \times \{0, 1\}}$. 
Consider the torus decomposition \cite{JS,Jo} of $E(K)$ and denote the outermost piece which contains $\partial E(K)$ by $X$. 
The existence of an essential annulus $A = g(S^1 \times [0, 1])$ in $E(K)$ implies that $X$ is Seifert fibered. 
Hence $K$ is a cabled or a composite knot. 
Moreover, we may further isotope $A$ so that $A \subset X$ and it is vertical (i.e. consisting of fibers of Seifert fibration). 
Thus $\partial A$ is a regular fiber. 
Hence $\gamma_1$ represents either a cabling slope 
(if $K$ is a cabled knot), 
or a meridional slope (if $K$ is a composite knot).  
Finally, assume for a contradiction that $\gamma_2 = \gamma_1^{-1}$. 
Then after abelianization the equation $g \gamma_1 g^{-1} = \gamma_1^{-1}$ implies that 
$\gamma_1$ is trivial in $H_1(E(K))$, i.e. 
$\gamma_1$ is a preferred longitude. 
However, 
a cabling slope and a meridional slope are not preferred longitude, 
a contradiction. 
So $\gamma_2=\gamma_1$. 
\end{proof}

\section{Finitely generated normal closures of slopes}
\label{rank}
In this section we prove Theorem~\ref{finitely generated}. 
The next proposition gives a classification of finitely generated, 
normal subgroups of infinite index of knot groups. 

\begin{proposition}
\label{finitely_generated_normal}
Let $N$ be a finitely generated, 
nontrivial, normal subgroup of infinite index of $G(K)$. 
Then either

\begin{itemize}
\item[(i)]
$E(K)$ is Seifert fibered (i.e. $K$ is a torus knot) and $N$ is a subgroup of Seifert fiber subgroup, 
the subgroup generated by a regular fiber of a Seifert fibration 
(i.e. $N$ is a subgroup of the center of the torus knot group $G$) or, 
\item[(ii)]
$E(K)$ fibers over $S^1$ with surface fiber $\Sigma$ and $N$ is a subgroup of finite index of $\pi_1(\Sigma)$.
\end{itemize}
\end{proposition}

\begin{proof}
This essentially follows from the classification of finitely generated, 
normal subgroups of infinite index of $3$--manifold groups \cite{HemJaco}, \cite[p.118 (L9)]{AFW}. 
Suppose for a contradiction that $N$ is neither (i) nor (ii). 
Then it follows from  \cite{HemJaco}, \cite[p.118 (L9)]{AFW} that 
$E(K)$ is the union of two twisted $I$-bundle over a compact connected (possibly non-orientable) surface $\Sigma$ 
and $N$ is a subgroup of finite index of $\pi_1(\Sigma)$. 
Then $G(K)$ is written as an extension
\[ \{1\} \rightarrow \pi_{1}(\Sigma) \rightarrow G(K) \rightarrow \mathbb{Z}_2 \ast \mathbb{Z}_2 \rightarrow \{1\}.\]
However, 
this would imply  $G(K)/[G(K), G(K)] \cong \mathbb{Z}$ has an epimorphism to $\mathbb{Z}_2 \oplus \mathbb{Z}_2$, 
the abelianization of 
$\mathbb{Z}_2 \ast \mathbb{Z}_2$, 
which is impossible. 
\end{proof}

\noindent\textbf{Proof of Theorem~\ref{finitely generated}.}
Let us assume $r$ is a finite surgery slope of $K$. 
Then $\pi_1(K(r)) = G(K) / \langle\!\langle r \rangle\!\rangle$ is finite. 
Hence $\langle\!\langle r \rangle\!\rangle$ is a subgroup of finite index of the finitely generated group $G(K)$,  
so it is finitely generated \cite[Corollary~2.7.1]{MKS}.

If $K$ is a torus knot $T_{p, q}$ and $r$ is a cabling slope $pq$, 
then $r$ is represented by a regular fiber $t$ in the Seifert fiber space $E(T_{p, q})$, 
and $\langle\!\langle pq \rangle\!\rangle$ is the infinite cyclic normal subgroup generated by $t$. 
This means that $\langle\!\langle pq \rangle\!\rangle \cong \mathbb{Z}$. 

To prove the converse, we suppose that $\langle\!\langle r \rangle\!\rangle$ is finitely generated. 
We divide into two cases depending upon $\langle\!\langle r \rangle\!\rangle$ has finite index in $G(K)$ or not. 
If it has finite index, then $G(K)/\langle\!\langle r \rangle\!\rangle = \pi_1(K(r))$ is a finite group, 
and hence $r$ is a finite surgery slope. 

If $\langle\!\langle r \rangle\!\rangle$ has infinite index in $G(K)$, 
then we have two possibilities described in Proposition \ref{finitely_generated_normal}. 
If we have the case (i) in Proposition \ref{finitely_generated_normal}, 
then $K$ is a torus knot and $r = pq$. 
Now suppose for a contradiction that the case (ii) in Proposition \ref{finitely_generated_normal} occurs. 
Then $\langle\!\langle r \rangle\!\rangle$ is a subgroup of finite index of the normal subgroup $\pi_{1}(\Sigma) (\subset G(K))$, 
in particular, 
$r$ lies in the Seifert surface subgroup $\pi_{1}(\Sigma)$. 
This implies  $r$ is a preferred longitude, i.e. $r = 0$. 
Since $\langle\!\langle 0 \rangle\!\rangle = \langle\!\langle \partial \Sigma \rangle\!\rangle \subset \pi_1(\Sigma)$
is a normal subgroup of $G(K)$, 
it is also normal in $\pi_1(\Sigma)$. 
Hence $\langle\!\langle \partial \Sigma \rangle\!\rangle$ is a normal subgroup of finite index of $\pi_1(\Sigma)$, 
and hence $\pi_1(\Sigma) / \langle\!\langle \partial \Sigma \rangle\!\rangle \cong \pi_1(\widehat \Sigma)$ would be a finite group, 
where $\widehat \Sigma$ is a closed orientable surface of genus $g(K) \ge 1$ obtained by capping off $\Sigma$ along $\partial \Sigma$. 
This is a contradiction. 
\qed

We close this section with the following. 

\begin{proposition}
\label{intersection_finitely_generated}
Let $K$ be a nontrivial knot and $r$ a slope of $K$ which is neither a finite surgery slope nor a  cabling slope $pq$ if $K$ is a torus knot $T_{p, q}$. 
Then for any normal subgroup $N$ of $G(K)$ 
the intersection $N \cap \langle\!\langle r \rangle\!\rangle$ is either trivial or not finitely generated. 
\end{proposition}

\begin{proof}
Theorem~\ref{finitely generated}, together with the assumption, shows that $\langle\!\langle r \rangle\!\rangle$ is not finitely generated. 
Let $M$ be the covering space of $E(K)$ associated to $\langle\!\langle r \rangle\!\rangle \subset G(K)$. 
Then $\langle\!\langle r \rangle\!\rangle$ is an infinitely generated $3$--manifold group $\pi_1(M)$.  
Let us write $H = N \cap \langle\!\langle r \rangle\!\rangle$, 
which is normal in $G(K)$, and hence normal in $\langle\!\langle r \rangle\!\rangle$. 
Assume for a contradiction that $H$ is nontrivial and finitely generated.   
Then \cite[Theorem~3.2]{Scott} shows that $H$ is infinite cyclic. 
This implies that $K$ is a torus knot $T_{p, q}$ \cite{GH} and $H$ is contained in the infinite cyclic normal subgroup 
$\langle\!\langle pq \rangle\!\rangle$ generated by a regular fiber of a Seifert fibration of $E(K) = E(T_{p, q})$ \cite[II.4.8.Lemma]{JS}. 
Hence $H \subset \langle\!\langle r \rangle\!\rangle \cap \langle\!\langle pq \rangle\!\rangle$. 
However, $\langle\!\langle r \rangle\!\rangle \cap \langle\!\langle pq \rangle\!\rangle$ 
would be trivial as we will prove in Proposition~\ref{reducing}(ii). 
This is a contradiction. 
Thus $N \cap \langle\!\langle r \rangle\!\rangle$ is not finitely generated. 
\end{proof}

\section{Finite family of normal closures of slopes and their intersection}
\label{intersection}

The goal of this section is to establish Theorem~\ref{thm:finite_family}.

For an element $h \in G(K)$, 
we denote its centralizer $\{g \in G(K)\: | \: gh=hg\}$ by $Z(h)$. 
We call $h \in G(K)$ a \textit{central element} of $G(K)$ 
if $Z(h) = G(K)$, 
and denote the \textit{center} of $G(K)$, 
the normal subgroup consisting of all the central elements, 
by $Z(K)$. 

Let $r$ be a slope of $K$. 
Then, throughout this section, 
we use $r$ to denote also a slope element $\gamma \in P(K)$ representing $r$. 
So $Z(r)$ means $Z(\gamma)$. 
(Note that $\gamma^{-1}$ also represents $r$ and $Z(\gamma) = Z(\gamma^{-1})$.)

Recall that $Z(r) = G(K)$ happens for some slope $r$ 
if and only if $K$ is a torus knot $T_{p, q}$ and $r = pq$; 
see \cite{BurdeZieschang} and \cite[Theorem~2.5.2]{AFW}.
Also, we remark that by Proposition \ref{conjugation} $Z(r)=P(K)$ unless $K$ is cabled and $r$ is the cabling slope, or $K$ is composite and $r$ is the meridional slope (we do not use this fact, though). 

If $K = T_{p, q}$ and $r = pq$, 
then given non-central element $g \in G(K)$, 
obviously $a g a^{-1} \in Z(r)$ for any $a \in G(K)$. 
Except this very restricted situation, 
we have:

\begin{lemma}
\label{centralizer}
Let $K$ be a nontrivial knot and $r$ a nontrivial slope of $K$.   
If $K$ is a torus knot $T_{p, q}$, we assume $r \neq pq$.
Then for every non-central element $g \in G(K)$, 
we can take an element $a \in G(K)$ so that $aga^{-1} \not \in Z(r)$.
\end{lemma}

\begin{proof}
If $g \not\in Z(r)$, 
then take $a = 1$ to obtain the desired conclusion. 
So in the following we assume $g \in Z(r)$. 
By a structure theorem of the centralizer of 3-manifold groups \cite[Theorem 2.5.1]{AFW}, 
either $Z(r)$ is an abelian group of rank at most two, 
or $Z(r)$ is conjugate to a subgroup of the fundamental group of a Seifert fibered piece of $E(K)$ with respect to the torus decomposition of $E(K)$ \cite{JS, Jo}.

First assume that $Z(r)$ is an abelian group of rank at most two. 
(In fact, $Z(r) = \mathbb{Z} \oplus \mathbb{Z}$, because $Z(r) \supset P(K)$.)
Assume, to the contrary that $aga^{-1} \in Z(r)$ for all $a \in G(K)$. 
Then the normal closure $N = \langle\!\langle g \rangle\!\rangle$ of $g$ in $G(K)$ is a nontrivial normal subgroup of $G(K)$ contained in $Z(r)$.
Thus $N$ is finitely generated. 
If $N$ has finite index in $G(K)$ then $G(K)$ is virtually abelian, 
which cannot happen for nontrivial knot groups since the knot group contains a free group of rank $\geq 2$, 
the fundamental group of the minimum genus Seifert surface. 
So $N$ is an abelian normal subgroup of infinite index of $G(K)$. 
By Proposition \ref{finitely_generated_normal}, 
either $K$ is a torus knot and $N$ is a subgroup of the center of $G(K)$, 
or $E(K)$ fibers over $S^{1}$ with torus fiber. 
In the former case, $g \in N$ is a central element, contradicting the choice of $g$.  
In the latter case $\partial E(K) = \emptyset$ so this cannot happen, either.

Next we assume that $Z(r)$ is not an abelian group of rank at most two. 
Then $K$ is a torus knot, or a satellite knot which has a Seifert fibered piece with respect to its torus decomposition. 
Assume first that $K$ is a torus knot $T_{p, q}$. 
Then $Z(r)$ is $\mathbb{Z} \oplus \mathbb{Z}$ since $r \ne pq$ (cf. \cite[Theorem 2.5.2]{AFW}).
This contradicts the assumption. 
Thus $K$ is a satellite knot whose exterior has a Seifert fibered piece  $M$ with respect to its torus decomposition.  
Let $T$ be an essential torus which is a member of the family of tori giving the torus decomposition of $E(K)$. 
Since $T$ is separating, $E(K)$ is expressed as $E_1 \cup _{T} E_2$; 
we may assume $M \subset E_1$. 
Then $G(K)$ is a nontrivial amalgamated product $G(K) = G_1\ast_{H}G_2$, 
where $H$ denotes the fundamental group of an essential torus $T$ and $G_i = \pi(E_i)$. 
The centralizer of $r$ is a conjugate to a subgroup of $\pi_1(M) \subset G_1$, 
thus there is an element $u \in G(K)$ such that $u Z(r) u^{-1} \subset G_1$.
Since $g \in Z(r)$, 
this implies $u g u^{-1} \in u Z(r) u^{-1} \subset G_1$. 
Let us write $w = u g u^{-1} \in G_1$. 

\begin{claim}
\label{torus knot}
There exists an element 
$v \in G(K)$ such that 
$v w v^{-1} \not\in G_1$. 
\end{claim}

\begin{proof}
We divide the argument into the following two cases. \\

\noindent
\textbf{Case 1.} There exists $f \in G_1$ such that $f w f^{-1} \in G_1 - H$.\\ 

We choose $v'$ so that $v' \in G_2 - H$ and let $v = v' f$.
Then by \cite[Corollary 4.4.2]{MKS}, 
the canonical form of $v w v^{-1} = v' (fwf^{-1}) v'^{-1}$ has length three, 
which is independent of a choice of right coset representatives. 
Hence,  
$v w v^{-1} \not\in G_1$.  \\

\noindent
\textbf{Case 2.} For every $f \in G_1$, $f w f^{-1} \in H$. \\

Let $N_1$ be the normal closure $\langle\!\langle  w \rangle\!\rangle_{G_1}$ of $w$ in $G_1$. 
Then $N_1 \subset H \cong \mathbb{Z} \oplus \mathbb{Z}$ is an abelian normal subgroup of $G_1$. 
Assume that  $N_1 \cong \mathbb{Z} \oplus \mathbb{Z}$ and it has finite index in $G_1$. 
Then \cite[Theorem 10.6]{Hem} shows that $E_1$ is either $S^1 \times S^1 \times [0, 1]$ or a twisted $I$--bundle over the Klein bottle. 
Either case cannot occur. 
So $N_1$ has infinite index in $G_1$. 
Then referring \cite{HemJaco} or \cite[p.118 (L9)]{AFW},  
$E_1$ would be the union of two twisted $I$--bundle over the Klein bottle, 
or a torus bundle over $S^{1}$. 
They are closed, a contradiction.  

Thus $N_1 \cong \mathbb{Z}$. 
By  \cite[p.118 (L9)]{AFW}, 
$E_1$ is Seifert fibered and $N_1$ is a Seifert fiber subgroup, a subgroup generated by a regular fiber of a Seifert fibration of $E_1$. 
Then $N_1$ is a central subgroup of $G_1$, 
and hence $w \in  \langle\!\langle w \rangle\!\rangle_{G_1} = N_1$ is central and 
$fwf^{-1}=w$ for all $f \in G_1$. 
This means that $N_1 = \langle w \rangle$. 

Now we show that there is an element $v \in G_2 \subset G(K)$ such that $v w v^{-1} \not \in G_1$. 
Assume to the contrary that for any element $v \in G_2$, $v w v^{-1} \in G_1$.  
Since $v \in G_2$ and $w \in H \subset G_2$, 
$v w v^{-1} \in G_2$. 
In particular, $v w v^{-1} \in G_1 \cap G_2 = H$. 
Let us consider the normal closure $N_2 = \langle\!\langle w \rangle\!\rangle_{G_2}$ of $w$ in $G_2$.  
Apply the above argument to see that $v w v^{-1}=w$ for all $v \in G_2$ and 
$N_2= \langle w \rangle$. 

This shows that $N_{1} = N_{2} =\langle w \rangle$. 
Since $w$ is central in both $G_1$ and $G_2$,  
$w$ is a central element in the entire group $G(K)$. 
However, this implies $g = u^{-1} w u = w$ is a central element, contradicting the choice of $g$. 
This completes a proof of Claim~\ref{torus knot}
\end{proof}

For $a = u^{-1}v u$, we have:
\[ a g a^{-1} 
= (u^{-1} v u) g (u^{-1} v u)^{-1}
= u^{-1} (v u g u^{-1} v^{-1}) u 
= u^{-1} (v w v^{-1}) u 
\not \in u^{-1} G_1 u. \]
Since $Z(r) \subset u^{-1} G_1 u$, 
we have $a g a^{-1} \not \in Z(r)$, 
as desired.
\end{proof}

Now we are ready to prove: 

\medskip

\begin{thm:finite_family}
Let $K$ be a nontrivial knot in $S^3$,  
and let $\{ r_1, \dots, r_n \}$ $(n\geq 2)$ be any finite family of slopes of $K$.  
If $K$ is  a torus knot $T_{p, q}$, 
we assume that $pq \not\in \{ r_1, \dots, r_n \}$. 
Then $\langle\!\langle r_1\rangle\!\rangle \cap \cdots \cap \langle\!\langle r_n\rangle\!\rangle$ is nontrivial.
Moreover, this subgroup is finitely generated if and only if all the $r_i$ are finite surgery slopes. 
\end{thm:finite_family}

\begin{proof} 
When $r_i$ is the trivial slope $\infty$ for some $1 \le i \le n$, 
then $\langle\!\langle r_i \rangle\!\rangle = G(K)$, 
so we assume that $r_i \neq \infty$ for all $1 \le i \le n$. 
We first show the nontriviality of the intersection.  
Let 
\[\Gamma_{0}(G(K)) \supset \Gamma_{1}(G(K)) \supset \cdots\]
be the lower central series of $G(K)$, 
defined inductively by 
\[\Gamma_{0}(G(K))=G(K)\quad \mathrm{and}\quad  \Gamma_{i}(G(K)) = [G(K),\Gamma_{i-1}(G(K))]\  (i>0).\]
Although it is known that $\Gamma_i(G(K)) = \Gamma_1(G(K)) = [G(K), G(K)]$ for $i \ge 1$ \cite[p.59]{Neu}, 
we use $\Gamma_i(G(K))$ for convenience in the following inductive construction of an element 
$g_{m} \in \Gamma_{m}(G(K))$ which belongs to 
$\bigcap_{i=1}^{m} \langle\!\langle r_i\rangle\!\rangle$ for $m=1,\ldots,n$.

For the case $m=1$, take $a_1 \in G(K)$ so that $a_1 \not \in Z(r_1)$ 
(by the assumption, $Z(r_1) \neq G(K)$ so such $a_1$ exists). 
Then $g_1=[r_1,\ a_1]= r_1 (a_{1}r_1^{-1}a_{1}^{-1})$ is a nontrivial element in $\Gamma_1(G(K)) = [G(K), G(K)]$ that belongs to $\langle\!\langle r_1 \rangle\!\rangle$ as desired.

Let us assume $m \ge 2$, 
and we have already found a nontrivial element $g_{m-1}$ in $\Gamma_{m-1}(G(K))$ such that 
$g_{m-1} \in \bigcap_{i=1}^{m-1} \langle\!\langle r_i\rangle\!\rangle$.

\begin{claim}
\label{non-central}
$g_{m-1}$ is not a central element in $G(K)$. 
\end{claim}

\begin{proof}
Assume for a contradiction that $g_{m-1}$ is a central element in $G(K)$. 
Then $K$ is a torus knot $T_{p, q}$ and $Z(K)$ is generated by a slope element $pq$, 
which is represented by regular fiber $t$ of the Seifert fibration of $E(T_{p, q})$; 
see \cite{BurdeZieschang} and \cite[Theorem~2.5.2]{AFW}. 
Thus $Z(K) = \langle pq \rangle$ and $g_{m-1} = t^{x}$ for some non-zero integer $x$. 
Since $g_{m-1} \in \Gamma_{m-1}(G(K)) \subset \Gamma_{1}(G(K)) = [G(K), G(K)]$, 
it should be trivial in $G(K) /  [G(K), G(K)] = H_1(E(K)) \cong \mathbb{Z}$. 
On the other hand, $t^{x} = (\mu^{pq} \lambda)^x$ becomes $\mu^{pqx} \in H_1(E(K))$, 
which is nontrivial, a contradiction. 
\end{proof}

Thus by Lemma \ref{centralizer} there is $a_{m} \in G(K)$ such that 
$a_m g_{m-1} a_m^{-1} \not \in Z(r_{m})$. 
Let us take 
\[g_m = [r_{m},\ a_m g_{m-1} a_m^{-1}] \in [G(K), \Gamma_{m-1}(G(K))] = \Gamma_m(G(K)).\] 
Since $a_m g_{m-1} a_m^{-1} \not \in Z(r_{m})$, $g_{m} \neq 1$. 
Obviously $g_m \in \langle\!\langle r_m \rangle\!\rangle$. 
Since $g_{m-1} \in \bigcap_{i=1}^{m-1} \langle\!\langle r_i\rangle\!\rangle$ 
and $\bigcap_{i=1}^{m-1} \langle\!\langle r_i\rangle\!\rangle$ is normal in $G(K)$, 
$a_m g_{m-1} a_m^{-1} \in  \bigcap_{i=1}^{m-1} \langle\!\langle r_i\rangle\!\rangle$ as well. 
Therefore $1 \neq g_m \in \bigcap_{i=1}^{m} \langle\!\langle r_i \rangle\!\rangle$. 

Next we determine when 
$\langle\!\langle r_1\rangle\!\rangle \cap \cdots \cap \langle\!\langle r_n\rangle\!\rangle$ 
is finitely generated.  
Assume that at least one of $r_i \in \{r_1, \dots, r_n\}$ is not a finite surgery slope. 
Without loss of generality, we may assume $r_1$ is not a finite surgery slope. 
By the assumption $r_1 \ne pq$ neither when $K = T_{p, q}$. 
Note that the above argument shows that $N = \bigcap_{i=1}^{n}\langle\!\langle r_i \rangle\!\rangle$ is nontrivial. 
Since $r_1$ is neither the cabling slope nor a finite surgery slope, 
Proposition~\ref{intersection_finitely_generated} shows that 
$N = N \cap \langle\!\langle r_1 \rangle\!\rangle$ is not finitely generated.

Conversely, 
assume that every $r_i$ is a finite surgery slope.
Then all $\langle\!\langle r_i \rangle\!\rangle$ are subgroup of finite index of $G(K)$, 
so is their intersections $\bigcap_{i=1}^{n}\langle\!\langle r_i \rangle\!\rangle$.  
Since $G(K)$ is finitely generated,  
so is its subgroup $\bigcap_{i=1}^{n}\langle\!\langle r_i \rangle\!\rangle$ of finite index.  
\end{proof}

In Theorem~\ref{thm:finite_family}, 
if $K$ is a torus knot,  
the assumption that $pq \not\in \{ r_1, \dots, r_n \}$ is essential. 
Actually, if $pq \in \{ r_1, \dots, r_n \}$, 
then $\langle\!\langle r_1\rangle\!\rangle \cap \cdots \cap \langle\!\langle r_n\rangle\!\rangle$ is generically trivial. 

\begin{proposition}
\label{reducing}
Let $K$ be a torus knot $T_{p, q}$ and $r$ a slope distinct from $pq$. 
\begin{itemize}
\item[(i)]
If $r$ is a finite surgery slope, 
then $\langle\!\langle pq \rangle\!\rangle \cap \langle\!\langle r \rangle\!\rangle\cong \mathbb{Z}$.
\item[(ii)]
If $r$ is not a finite surgery slope, 
then $\langle\!\langle pq \rangle\!\rangle \cap \langle\!\langle r \rangle\!\rangle = \{ 1 \}$. 
\end{itemize}
\end{proposition}

\begin{proof}
Recall that the slope $pq$ is represented by a regular fiber $t$ of the Seifert fiber space $E(T_{p, q})$, 
and hence the infinite cyclic normal subgroup generated by $t$ coincides with $\langle\!\langle pq \rangle\!\rangle$. 
Hence $\langle\!\langle pq \rangle\!\rangle \cap \langle\!\langle r \rangle\!\rangle = \langle t^k \rangle \subset \langle t \rangle \cong \mathbb{Z}$ for some integer $k\ge 0$. 
Hence, $\langle\!\langle pq \rangle\!\rangle \cap \langle\!\langle r \rangle\!\rangle$ is either trivial or infinite cyclic. 

(i)
If $r$ is a finite surgery slope, 
$t$ has a finite order in the finite group $G(K)/\langle\!\langle r\rangle\!\rangle$.
Thus there exists an integer $ k' \ge 1$ such that $t^{k'} \in \langle\!\langle r\rangle\!\rangle$ in $G(K)$.
Thus 
$\langle\!\langle pq \rangle\!\rangle \cap \langle\!\langle r \rangle\!\rangle$ is nontrivial, 
and hence it is infinite cyclic. 

(ii)
Recall that $G(K) = \langle a, b\ |\ a^p = b^q \rangle$ and $t = a^p = b^q$. 
Thus $t^k = (a^p)^k = a^{pk} \in \langle\!\langle r \rangle\!\rangle$. 
This means that $a^{pk} = 1$ in $\pi_1(K(r))$. 
Since $r$ is neither a reducing surgery slope nor a finite surgery slope, 
$\pi_1(K(r))$ has no torsion element. 
Therefore $a = 1$ in $\pi_1(K(r))$, unless $k=0$. 
However, this implies that $G(K)$ is finite cyclic, 
a contradiction. 
Thus $k=0$, so $\langle\!\langle pq \rangle\!\rangle \cap \langle\!\langle r \rangle\!\rangle=\{1\}$.
\end{proof}

\begin{corollary}
\label{cor:finite_family}
Let $K$ be a hyperbolic knot in $S^3$ which is not the $(-2, 3, 7)$--pretzel knot.  
Then for any finite family of slopes $r_1, \dots, r_n \in \mathbb{Q}$ with $n > 2$, 
the subgroup $\cap_{i = 1}^n \langle\!\langle r_i \rangle\!\rangle$ is not finitely generated.
\end{corollary}

\begin{proof}
Since $K$ is a hyperbolic knot, 
\cite[Theorem~1.4]{NZ_finite} shows that it has at most three nontrivial finite surgeries, 
and except when $K$ is the $(-2, 3, 7)$--pretzel knot, it has at most two such surgeries. 
Thus the result follows from Theorem~\ref{thm:finite_family} immediately. 
\end{proof}

\begin{example}
\label{P(-2,3,7)}
Let $K$ be the $(-2, 3, 7)$--pretzel knot. 
Then it has three nontrivial finite surgery slopes: $17, 18$ and $19$. 
Hence, $\langle\!\langle 17 \rangle\!\rangle \cap \langle\!\langle 18 \rangle\!\rangle \cap 
\langle\!\langle 19 \rangle\!\rangle$ is a finitely generated nontrivial normal subgroup of $G(K)$. 
\end{example}

\section{Chains of normal closures of slopes}

Applying Propositions~\ref{inclusion} and \ref{non_finite_finite} we have: 

\medskip

\begin{thm:no_infinite_chain}
Let $K$ be a non-torus knot in $S^3$. 
If $\langle\!\langle r_1 \rangle\!\rangle \supset \cdots \supset \langle\!\langle r_n \rangle\!\rangle$
for mutually distinct slopes $r_1, \dots, r_n \in \mathbb{Q}$,  
then $n \le 2$.  
In particular, there is no infinite descending chain nor ascending chain of normal closures of slopes. 
\end{thm:no_infinite_chain}

\begin{proof}
Assume for a contradiction that 
we have $\langle\!\langle r_1 \rangle\!\rangle \supset \langle\!\langle r_2 \rangle\!\rangle \supset \langle\!\langle r_3 \rangle\!\rangle$ 
for mutually distinct slopes $r_1, r_2$ and $r_3$. 
Then by Proposition~\ref{inclusion}, $r_1$ and $r_2$ are both finite surgery slopes. 
However, this contradicts Proposition~\ref{non_finite_finite}. 
\end{proof}

As shown in Proposition~\ref{inclusion}, 
the inclusion $\langle\!\langle r \rangle\!\rangle \supset \langle\!\langle r' \rangle\!\rangle$
can occur only when $r$ is a finite surgery slope. 
In fact, for a given finite surgery slope $r$, we can find infinitely many slopes $r_k$ so that 
$\langle\!\langle r \rangle\!\rangle \supset \langle\!\langle r_k \rangle\!\rangle$.

\begin{proposition}
\label{finite_finite}
Let $K$ be a nontrivial knot with finite surgery slope $m/n\in \mathbb{Q}$.
Let $f$ be the order of the meridian $\mu$ in the finite group $G(K)/\langle\!\langle m/n \rangle\!\rangle$.
Then
\[
\langle\!\langle m/n \rangle\!\rangle \supset \langle\!\langle m/kn \rangle\!\rangle
\]
for  any non-zero integer $k$ such that $\gcd(k,m)=1$ and $k\equiv 1\pmod{f}$.
If $m/n$ is a cyclic surgery slope, then the last condition $k\equiv 1\pmod{f}$ is redundant.
\end{proposition}

\begin{proof}
In $G(K)/\langle\!\langle m/n \rangle\!\rangle$,
we have $\mu^m\lambda^n=1$, hence $\lambda^n=\mu^{-m}$.
Then the slope element $\mu^m\lambda^{kn}$
is equal to $\mu^{m(1-k)}$ there.
If $m/n$ is a cyclic surgery slope, then $\mu^m=1$, since $G(K)/\langle\!\langle m/n \rangle\!\rangle=\mathbb{Z}_{|m|}$.
Otherwise, $\mu^{k-1}=1$, since $k\equiv 1\pmod{f}$ and $\mu^{f} = 1$ in $G(K)/\langle\!\langle m/n \rangle\!\rangle$. 
Thus $\mu^m\lambda^{kn} = \mu^{m(1-k)} = 1$ in $G(K) / \langle\!\langle m/n \rangle\!\rangle$, 
i.e. $\mu^m\lambda^{kn} \in \langle\!\langle m/n \rangle\!\rangle$. 
Hence $\langle\!\langle m/kn \rangle\!\rangle \subset \langle\!\langle m/n \rangle\!\rangle$.  
\end{proof}

\begin{example}
Among hyperbolic Montesinos knots, the $(-2,3,7)$-pretzel knot
and $(-2,3,9)$-pretzel knot are the only ones that admit nontrivial finite surgeries \cite{IJ}.
To determine the order of the meridian in the resulting finite group, 
we used presentations given in \cite{Na} and \texttt{GAP} (Groups, Algorithms, Programming) \cite{GAP}.

(1) Let $K$ be the $(-2, 3, 7)$-pretzel knot. 
Then $18$ and $19$ are cyclic surgery slopes. 
Hence, by Proposition~\ref{finite_finite}, 
for any non-zero integer $k, k'$ such that $\gcd(k, 18)=1$ and $\gcd(k', 19)=1$,
\[
\langle\!\langle 18 \rangle\!\rangle \supset \langle\!\langle 18/k \rangle\!\rangle,\quad 
\langle\!\langle 19 \rangle\!\rangle \supset \langle\!\langle 19/k' \rangle\!\rangle. 
\]
Also, 
we may observe that the slope $17$ is a finite, non-cyclic surgery slope using Montesinos trick \cite{Mon}. 
Precisely, $K(17)$ is a Seifert fibered manifold with Seifert invariant $S^2(0; 1/3, -2/5, -1/2)$; 
see also \cite{EM}, \cite[Proposition~5.6]{EM2} and \cite[4.2]{DEMM} for some corrections of mistakes in Proposition~5.6 in \cite{EM2}.  
Hence its fundamental group is $I_{120}\times \mathbb{Z}_{17}$, 
where $I_{120}$ is the binary icosahedral group of order 120.
In the finite group $G(K)/\langle\!\langle 17 \rangle \!\rangle$ of order $2040$,
the order of the meridian is $170$.
Hence for any non-zero integer $k$ such that $\gcd(k,17)=1$ and $k\equiv 1\pmod{170}$,
\[
\langle\!\langle 17 \rangle\!\rangle \supset \langle\!\langle 17/k \rangle\!\rangle.
\]
(2) Let $K$ be the $(-2, 3, 9)$-pretzel knot.
It has two non-cyclic finite surgery slopes $22$ and $23$. 
Seifert invariants of Seifert fibered manifolds obtained by $22$ and $23$ surgeries on $K$ were given by \cite{BH}, 
but it contains miscalculations, so we give their corrections below. 
Using Montesinos trick, 
we observe that $K(22)$ is a Seifert fibered manifold with Seifert invariant $S^2(0; 1/2, -1/4, 2/3)$.  
Its fundamental group is $O_{48}\times \mathbb{Z}_{11}$, 
where $O_{48}$ is the binary octahedral group of order $48$.
In the finite group $G(K)/\langle\!\langle 22\rangle \!\rangle$ of order $528$, 
the order of the meridian is $44$. 
Similarly, we observe that $K(23)$ is a Seifert fibered manifold with Seifert invariant $S^2(0; 1/2, 2/3, -2/5)$,  
whose fundamental group is $I_{120}\times \mathbb{Z}_{23}$ of order $2760$. 
The order of the meridian is $138$ in this group.
Hence for any non-zero integers $k$, $k'$ such that $\gcd(k,528)=1$ and $k\equiv 1\pmod{44}$,
$\gcd(k',2760)=1$ and $k'\equiv 1 \pmod{138}$, 
\[
\langle\!\langle 22 \rangle\!\rangle \supset \langle\!\langle 22/k \rangle\!\rangle,\quad 
\langle\!\langle 23 \rangle\!\rangle \supset \langle\!\langle 23/k' \rangle\!\rangle. 
\]
\end{example}

In the rest of this section, we focus on normal closures of slopes for torus knots.

\begin{theorem}
\label{no_ascending}
Let $K$ be a torus knot $T_{p,q}$. 
Then there is no infinite ascending chain 
$\langle\!\langle r \rangle\!\rangle \subset  \langle\!\langle r_1 \rangle\!\rangle \subset \langle\!\langle r_2 \rangle\!\rangle \subset \cdots$. 
\end{theorem}

\begin{proof}
Suppose for a contradiction that there is an infinite ascending chain
\[\langle\!\langle r \rangle\!\rangle \subset  \langle\!\langle r_1 \rangle\!\rangle \subset \langle\!\langle r_2 \rangle\!\rangle \subset \cdots.\] 
By Proposition~\ref{inclusion}, all $r_i$ but $r$ are finite surgery slopes.
Let $r_i=m_i/n_i$ for $i\ge 1$.
We may assume that $m_i > 0$.
Since $\langle\!\langle r_i \rangle\!\rangle \subset  \langle\!\langle r_{i+1} \rangle\!\rangle$,  
we have a canonical epimorphism 
\[G(K) /  \langle\!\langle r_i \rangle\!\rangle \to G(K) / \langle\!\langle r_{i+1}\rangle\!\rangle.\]  
This induces an epimorphism
\[H_1(K(m_i/n_i )) \cong \mathbb{Z}_{m_i} \to H_1(K(m_{i+1}/n_{i+1} )) \cong \mathbb{Z}_{m_{i+1}}.\] 
Thus we have an infinite sequence 
$m_1 \ge m_2 \ge \cdots \ge m_i \ge m_{i+1} \ge \cdots$, 
where $m_i$ is divided by $m_{i+1}$. 
Hence there is a constant $N > 0$ such that 
$m_{i+1} = m_i$ for $i \ge N$; 
we denote this constant $m$. 
As stated in Case 2 of the proof of Theorem \ref{thm:peripheral Magnus property},
$|pqn_i-m|\in \{1,2,3,4,5\}$.
This implies that there are only finitely many possibilities for $n_i$, a contradiction.
\end{proof}

\begin{theorem}
\label{infinite_descending_sequence}
Let $K$ be a torus knot $T_{p,q}\ (p>q\ge 2)$.
For each finite surgery slope $r\in \mathbb{Q}$,
there exists an infinite descending chain
\[
\langle\!\langle r \rangle\!\rangle \supset \langle\!\langle r_1 \rangle\!\rangle\supset 
\langle\!\langle r_2 \rangle\!\rangle\supset \cdots.
\]
\end{theorem}

\begin{proof}
Assume first that $r$ is a cyclic surgery slope.
Then $r=pq+\frac{1}{n}$ for some non-zero integer $n$.
Set $r_1=pq+\frac{1}{(pq+1)n+1}$.
We show that 
$\langle\!\langle r \rangle\!\rangle \supset \langle\!\langle r_1 \rangle\!\rangle$. 

Using standard meridian-longitude pair $(\mu,\ \lambda)$ of $K$,  
$r$ is represented by $\mu^{pqn+1}\lambda^n$ 
and $r_1$ is represented by $\mu^{pq((pq+1)n+1)+1}\lambda^{(pq+1)n+1}$. 
Since $G(K)/\langle\!\langle r \rangle\!\rangle$ is abelian, 
the longitude $\lambda$ is trivial in $G(K)/\langle\!\langle r\rangle\!\rangle$, 
and hence, 
\[
\mu^{pqn+1}\lambda^n = \mu^{pqn+1},
\]   
\[
\mu^{pq((pq+1)n+1)+1}\lambda^{(pq+1)n+1} = \mu^{pq((pq+1)n+1)+1} = \mu^{(pq+1)(pqn+1)}
\]
in $G(K)/\langle\!\langle r\rangle\!\rangle$. 
Since $ \mu^{pqn+1} = 1$ in $G(K)/\langle\!\langle r\rangle\!\rangle$, 
so is $\mu^{(pq+1)(pqn+1)} = 1$. 
Thus $r_1 \in \langle\!\langle r\rangle\!\rangle$. 
We remark that $r_1$ remains a cyclic surgery slope.
Repeat this process to obtain a desired infinite descending chain.

Next assume that $r$ is a non-cyclic finite surgery slope $m/n$.
This is possible only for $(p,q,|pqn-m|)=(p,2,2)$, $(3,2,3)$, $(3,2,4)$, $(3,2,5)$,
$(5,2,3)$, $(5,3,2)$.
Put $d = pqn-m$, 
and let $(p,q,|d|)$ be one of these triples.
Set $r_1= \frac{m_1}{n_1} =  \frac{m(f+1)+df}{n(f+1)}$, 
where $f$ is the order of the meridian $\mu$ in the finite group $G(K)/\langle\!\langle r\rangle\!\rangle$.
Then we show that 
$\langle\!\langle r \rangle\!\rangle \supset \langle\!\langle r_1 \rangle\!\rangle$. 
In $G(K)/\langle\!\langle r \rangle\!\rangle$, $\mu^m\lambda^n=1$, so $\lambda^n=\mu^{-m}$.
Hence the slope element $\mu^{m_1}\lambda^{n_1} = \mu^{m(f+1)+df}\lambda^{n(f+1)}$ is equal to
$\mu^{m(f+1)+df-m(f+1)}=\mu^{df}=1$ in $G(K)/\langle\!\langle r \rangle\!\rangle$, 
so the slope element $\mu^{m_1}\lambda^{n_1}$ representing $r_1$  belongs to $\langle\!\langle r\rangle\!\rangle$. 
This means that $\langle\!\langle r_1 \rangle\!\rangle \subset \langle\!\langle r \rangle\!\rangle$. 
We remark that $pqn_1 - m_1 = d = pqn-m$. 
Then applying the above argument  to $r_2 = \frac{m_2}{n_2} = \frac{m_1(f_1 + 1) + d f_1}{n_1(f_1 + 1)}$, 
where $f_1$ is the order of $\mu$ in $G(K) / \langle\!\langle r_1 \rangle\!\rangle$, 
we see that $r_2$ belongs to $\langle\!\langle r_1 \rangle\!\rangle$. 
Repeat this process to obtain an infinite descending chain. 
\end{proof}

Theorem~\ref{thm:descending_sequence} follows from Theorems~\ref{no_ascending} and 
\ref{infinite_descending_sequence}.

\begin{proposition}
\label{cyclic_commutator}
Let $K$ be a torus knot $T_{p, q}$ $(p > q \ge 2)$. 

\begin{itemize}
\item[(i)]
For the infinite descending chain \[
\langle\!\langle pq+\frac{1}{n} \rangle\!\rangle \supset \langle\!\langle pq+\frac{1}{(pq+1)n+1} \rangle\!\rangle\supset 
 \cdots
\]
consisting of cyclic surgery slopes,
the intersection of these normal closures is the commutator subgroup $[G(K),G(K)]$ of $G(K)$.

\item[(ii)]
For the set of cyclic surgery slopes $\mathcal{S}$, 
$\bigcap_{r \in \mathcal{S}} \langle\!\langle r \rangle\!\rangle = [G(K), G(K)]$. 
\end{itemize}
\end{proposition}

\begin{proof}
(i) 
For simplicity, we denote the above sequence by $N_1\supset N_2 \supset \cdots$.
Then each quotient $G(K)/N_i$ is cyclic, so $[G(K),G(K)]\subset N_i$.
Hence $[G(K),G(K)]\subset \bigcap_{i = 1}^{\infty} N_i$.
Conversely, 
take an element $g\not \in [G(K),G(K)]$.
Then $g=\mu^k h$, where $\mu$ is a meridian, $h\in [G(K),G(K)]$ and $k\ne 0$.
However,
there exists $j$ such that  the order of the cyclic group $G(K)/N_j$ is bigger than $|k|$.
Then the epimorphism $G(K)\to G(K)/[G(K),G(K)]\to G(K)/N_j$ sends $g$ to a nontrivial element.
This means $g\not\in N_j$.
Hence $[G(K),G(K)]=\bigcap_{i = 1}^{\infty} N_i$.

(ii) 
Since $G(K)/ \langle\!\langle r \rangle\!\rangle$ is cyclic, 
$[G(K),G(K)]\subset \langle\!\langle r \rangle\!\rangle$ for any $r \in \mathcal{S}$. 
Thus $[G(K),G(K)] \subset \bigcap_{r \in \mathcal{S}} \langle\!\langle r \rangle\!\rangle$. 
Conversely, $\bigcap_{r \in \mathcal{S}} \langle\!\langle r \rangle\!\rangle \subset \bigcap_{i = 1}^{\infty} N_i = [G(K), G(K)]$. 
Hence $\bigcap_{r \in \mathcal{S}} \langle\!\langle r \rangle\!\rangle = [G(K), G(K)]$. 
\end{proof}

\begin{question}
For the infinite descending chain 
\[
\langle\!\langle r \rangle\!\rangle \supset \langle\!\langle r_1 \rangle\!\rangle\supset 
\langle\!\langle r_2 \rangle\!\rangle\supset \cdots
\]
consisting of non-cyclic finite surgery slopes in Theorem \ref{infinite_descending_sequence},
what is the intersection of these normal closures?
\end{question}

\section*{Funding}
This work was supported by Japan Society for the Promotion of Science,
KAKENHI Grant Numbers [JP15K17540 and JP16H02145 to I.T.,
JP26400099 to K. M., JP16K05149 to M.T.]; and
Joint Research Grant of Institute of Natural Sciences at Nihon University for 2017, 2018 to K. M. 

\medskip

\noindent
\textbf{Acknowledgements}{
We would like to thank the referee for a careful reading and useful comments,
and Toshio Sumi for some help of \texttt{GAP} at the beginning of this project and 
Ederson Dutra for useful comments.

\end{document}